\documentclass[reqno]{amsart}
\usepackage{latexsym}
\usepackage{amssymb,amsthm,amsmath,color}
\usepackage[dvips]{graphicx}

\theoremstyle{plain}
\newtheorem{theorem}{Theorem}
\newtheorem{lemma}[theorem]{Lemma}

\theoremstyle{definition}
\newtheorem{definition}[theorem]{Definition}
\theoremstyle{remark}
\newtheorem{remark}[theorem]{Remark}

\def\d#1{{#1\kern-0.4em\char"16\kern-0.1em}}
\def\D#1{{\raise0.2ex\hbox{-}\kern-0.4em #1}}
\newcounter{zd}

\newcounter{zdr}[subsection]

\newcommand{\eps}{\varepsilon}

\def\Dscon{\relbar\joinrel\dscon}
\def\povrhsk#1{\smash{
        \mathop{\;\Dscon\;}\limits^{#1}}}

\def\R{I\!\!R}

\def\N{I\!\!N}

\def\F{{\cal F}}
\def\pa{\partial}

\def\cal{\mathcal}

\def\dscon{\relbar\joinrel\rightharpoonup}
\def\Dscon{\relbar\joinrel\dscon}

\def\povrhsk#1{\smash{
        \mathop{\;\Dscon\;}\limits^{#1}}}

\let\ph=\varphi

\begin{document}

\title[$H$-measures and fractional conservation laws]
{A generalization of $H$-measures and application on purely
fractional scalar conservation laws}

\author{ D.~ Mitrovic}\thanks{Permanent address od D.M. is University of Montenegro, Faculty of Mathematics, Cetinjski put bb, 81000 Podgorica, Montenegro}
\address{ Darko Mitrovic, University of Bergen, Johannes Bruns gate 12, 5020 Bergen}
 \email{  matematika@t-com.me}
\author{I.~Ivec}
 \address{ Ivan Ivec, University of Zagreb, Faculty of Mathematics, Bijenicka cesta 30, 10000 Zagreb, Croatia}
 \email{ivan.ivec@gmail.com}

\maketitle

\begin{abstract}
We extend the notion of $H$-measures on test functions defined on
$\R^d\times P$, where $P\subset \R^d$ is an arbitrary compact simply
connected Lipschitz manifold such that there exists a family of
regular nonintersecting curves issuing from the manifold and
fibrating $\R^d$. We introduce a concept of quasi-solutions to
purely fractional scalar conservation laws and apply our extension
of the $H$-measures to prove strong $L^1_{loc}$ precompactness of
such quasi-solutions.
\end{abstract}

\section{Introduction}

Suppose that we wish to solve a nonlinear PDE which we write
symbolically as $A[u]=f$, where $A$ denotes a given nonlinear
operator. One of usual ways to do it is to approximate the PDE by a
collection of "nicer" problems $A_k[u_k]=f_k$, where $(A_k)$ is a
sequence of operators which is somehow close to $A$. Then, we try to
prove that the sequence $(u_k)$ converges toward a solution to the
original problem $A[u]=f$. The overall impediment is of course
nonlinearity which prevents us from obtaining necessary uniform
estimates on the sequence $(u_k)$. The typical situation is the
following.

Let $\Omega$ be an open set in $\R^d$, and let $(u_k)$ be a bounded
sequence in $L^2(\Omega)$ converging in the sense of distributions
to $u\in L^2(\Omega)$. In order to prove that $u$ is a solution to
$A[u]=f$, we need to prove that $(u_k)$ converges strongly to $u$,
say, in $L^1_{loc}(\Omega)$ (often situation in conservation laws;
see e.g. \cite{MA1, Dpe, pan_arma}). One of the ways is to consider
the sequence $\nu_k=|u_k-u|^2$ bounded in the space of Radon
measures ${\cal M}(\R^d)$. Since it is bounded, there exists a
measure $\nu$ such that $\nu_k\rightharpoonup \nu$ along a
subsequence in ${\cal M}(\R^d)$. The support of $\nu$ is the set of
points in $\Omega$ near which $(u_k)$ does not converge to $u$ for
the strong topology of $L^2(\R^d)$. The measure $\nu$ is called a
defect measure and it was systematically studied by P.L.Lions. For
instance, if we are able to prove that $\nu$ is equal to zero out of
a negligible set, then $(u_k)$ will $L^2$-strongly converge toward
$u$ on a set large enough to state that $u$ is a solution to
$A[u]=f$. Such method is called the concentrated compactness method
\cite{Lio1, Lio2}.

A shortcoming of the latter defect measure is that they are not
sensitive to oscillation corresponding to different frequencies. For
instance, consider the sequence $(u_k(x))_{k\in
\N}=(exp(ikx\xi))_{k\in \N}$, where $i$ is the imaginary unit,
$\xi\in \R^d$ is a fixed vector, and $x\in \R^d$ is a variable. The
sequence is bounded which implies that it is bounded in
$L^2(\Omega)$ for any bounded $\Omega\subset \R^d$. Furthermore, it
is well known that $u_k\rightharpoonup 0$ in the sense of
distributions but $(u_k)$ does not converge strongly in $L^p_{loc}$
for any $p>0$. On the other hand, the defect measure $\nu$
corresponding to the sequence $(u_k)$ is the Lebesgue measure for
any $\xi\in \R^d$ (and $\xi$ determines the frequency of the rapidly
oscillating sequence $(u_k)$).

Step forward in this direction was made at the beginning of 90's
when L.Tartar \cite{Tar} and P.Gerard \cite{Ger} independently
introduced the $H$-measures (microlocal defect measures). They are
given by the following theorem:

\begin{theorem}\cite{Tar}
\label{tbasic1} If $(u_n)=((u_n^1,\dots, u_n^r))$ is a sequence in
$L^2(\R^d;\R^r)$ such that $u_n\rightharpoonup 0$ in
$L^2(\R^d;\R^r)$, then there exists its subsequence $(u_{n'})$ and a
positive definite matrix of complex Radon measures
$\mu=\{\mu^{ij}\}_{i,j=1,\dots,r}$ on $\R^d\times S^{d-1}$ such that
for all $\varphi_1,\varphi_2\in C_0(\R^d)$ and $\psi\in C(S^{d-1})$:
\begin{equation}
\label{basic1}
\begin{split}
\lim\limits_{n'\to \infty}\int_{\R^d}&(\varphi_1
u^i_{n'})(x)\overline{{\cal A}_\psi(\varphi_2
u^j_{n'})(x)}dx=\langle\mu^{ij},\varphi_1\overline{\varphi_2}\psi
\rangle\\&= \int_{\R^d\times
S^{d-1}}\varphi_1(x)\overline{\varphi_2(x)}\psi(\xi)d\mu^{ij}(x,\xi),
\ \ i,j=1,\dots,r,
\end{split}
\end{equation}where ${\cal A}_\psi$ is a multiplier operator with the symbol $\psi\in C(S^{d-1})$.
\end{theorem}

The complex matrix Radon measure $\{\mu^{ij}\}_{i,j=1,\dots, r}$
defined in the previous theorem we call the {\em $H$-measure}
corresponding to the subsequence $(u_{n'})\in L^2(\R^d;\R^r)$.

The $H$-measures describe a loss of strong $L^2$ compactness for the
corresponding sequence $(u_n)\in L^2(\R^d;\R^r)$. In order to
clarify the latter, assume that we are dealing with one dimensional
sequence $(u_n)$ (this means that $r=1$). Then, notice that, by
applying the Plancherel theorem, the term under the limit sign in
Theorem \ref{tbasic1} takes the form
\begin{equation}
\label{ns30} \int_{\R^d} \widehat{\ph_1 u_{n'}}
\overline{\psi\widehat{\ph_2 u_{n'}}}\,d\xi\;,
\end{equation}where by $\hat u(\xi) = ({\cal F}u)(\xi) = \int_{\R^d} e^{-2\pi
ix\cdot\xi} u(x)\,dx$ we denote the Fourier transform on $\R^d$
(with the inverse $(\bar{\cal F}v)(x) := \int_{\R^d} e^{2\pi
ix\cdot\xi} v(\xi)\,d\xi$). Now, it is not difficult to see that if
$(u_n)$ is strongly convergent in $L^2$, then the corresponding
H-measure is trivial. Conversely, if the H-measure is trivial, then
$u_n\longrightarrow 0$ in $L^2_{loc}(\R^d)$ (see \cite{Ant}).

One of constraints in using the $H$-measures concept is that the
symbols of the defining  multipliers appearing in \eqref{basic1} are
defined on the unit sphere. This makes the concepts adapted for
usage basically only on hyperbolic problems (see e.g. \cite{MA1,
Ger, pan_jhde} and exceptions \cite{Sazh, JMSpa}). The reason for
the mentioned confinement lies in the lemma which provides linearity
of the integral on the right-hand side of \eqref{basic1}. This is so
called first commutation lemma and is stated as follows:

\begin{lemma}\cite[Lemma 1.7]{Tar}
\label{scl} (First commutation lemma)  Let $a\in C(S^{d-1})$ and
$b\in C_0(\R^d)$. Let ${\cal A}$ be a multiplier operator with the
symbol $a$, and $B$ be an operator of multiplication given by the
formulae:
\begin{align*}
&\F({\cal A}u)(\xi)=a\big(\frac{\xi}{|\xi|}\big)\F(u)(\xi) \ \ a.e. \ \ \xi\in \R^d,\\
 &Bu(x)=b(x)u(x) \ \ a.e. \ \ x\in \R^d,
\end{align*} where ${\cal F}$ is the Fourier transform.
Then $C={\cal A}B-B{\cal A}$ is a compact operator from $L^2(\R^d)$
into $L^2(\R^d)$.
\end{lemma} As we can see, the symbol $a$ given above is defined on
the unit sphere. Recently, in \cite{Ant2} the first commutation
lemma was extended for symbol $a$ defined on the parabolic manifold
$P=\{(\tau,\xi)\in \R\times \R^d:\;\tau^2+|\xi|^4=1 \}$, and then,
in an analog fashion, in \cite{JMSpa} on the ultra-parabolic
manifold $UP=\{(\tilde{\xi},\bar{\xi})\in \R^k\times
\R^{(d-k)}:\;|\tilde{\xi}|^2+|\hat{\xi}|^4=1 \}$. This enabled the
authors of \cite{Ant2} and \cite{JMSpa} to replace in Theorem
\ref{tbasic1} the sphere $S^{d-1}$ by $P$ and $UP$, respectively.

We have noticed that the proof of the first commutation lemma relies
only on the fact that if  we project any compact set $K$ on the
sphere along the rays issuing from the origin, the projection will
be smaller as the distance of $K$ from the origin is larger.
Furthermore, it is clear that we do not need to project the set
$K\subset \R^d$ along the rays -- the projection curves can be
arbitrary smooth nonintersecting curves fibrating the space (see
Figure 1). We will use this observation in Section 2 to replace the
sphere $S^{d-1}$ in Theorem \ref{tbasic1} by an arbitrary compact
simply connected Lipschitz manifold such that there exists a family
of regular nonintersecting curves issuing from the manifold and
fibrating $\R^d$.

{ In Section 3, we consider the fractional scalar conservation law:
\begin{equation}
\label{eq_frac} \sum\limits_{k=1}^d\pa^{\alpha_k}_{x_k} f_k(x,u)=0,
\end{equation} where $\alpha_k\in (0,1]$, $f_k\in BV(\R^d;C^1(\R))$,
$k=1,\dots,d$. We start by introducing a notion of quasi-solutions
to \eqref{eq_frac} which are basically functions $u\in
L^\infty(\R^d)$ such that for every $\lambda\in \R$, the operator
$\sum\limits_{k=1}^d\pa^{\alpha_k}_{x_k} {\rm
sgn}(u\!-\!\lambda)(f_k(x,u)\!-\!f_k(x,\lambda))$ is compact as
mapping from $W^{1,\infty}(\R^d)$ to $L_{loc}^1(\R^d)$ (for a more
precise definition see Definition \ref{q-sol}). In the case of the
classical scalar conservation law, the latter operator is nothing
but the entropy defect measure. The main result of the section is
the fact that under a genuine nonlinearity conditions (see
Definition \ref{non=deg}), any bounded sequence of quasi-solutions
to \eqref{eq_frac} is strongly $L^1_{loc}$-precompact.}

\begin{figure}[htp]
\begin{center}
  \includegraphics[width=2in]{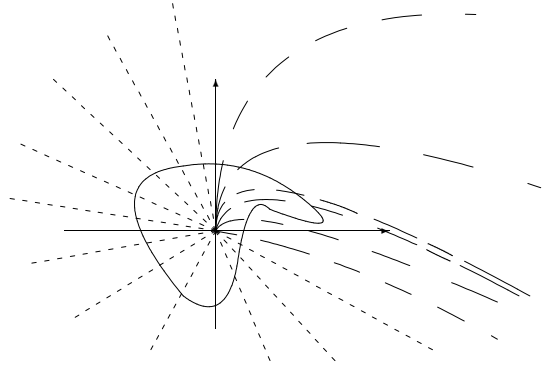}\\
  \caption{The manifold $P$ is represented by normal line. Fibres are dashed.
  Notice that a fibre must not intersect $P$ twice.}
  \label{sl1}
  \end{center}
\end{figure}

\section{The $H$-measures revisited}

In order to improve Theorem \ref{tbasic1}, we need a new variant of
the first commutation lemma. To introduce it, we need the following
operators. Let ${\cal A}$ be a multiplier operator with a symbol
$a\in C(\R^d)$, and $B$ be an operator of multiplication by a
function $b\in C_0(\R^d)$, given by the formulae:
\begin{align}
\label{oper_1}
&\F({\cal A}u)(\xi)=a(\xi)\F(u)(\xi) \ \ a.e. \ \ \xi\in \R^d,\\
\label{oper_2} &Bu(x)=b(x)u(x) \ \ a.e. \ \ x\in \R^d,
\end{align} where ${\cal F}$ is the Fourier transform.

Following the proof from \cite[Lemma 1.7]{Tar}, we shall see in
Lemma \ref{scl'} that the commutator $C={\cal A}B-B{\cal A}$ is a
compact operator from $L^2(\R^d)$ into $L^{2}(\R^d)$ if $a\in
L^\infty(\R^d)$ satisfies the following condition (see \cite[Lemma
28.2]{tar_book}): { \begin{equation} \label{uvjet} (\forall
R>0)(\forall \eps
>0)(\exists r>0)\; |\xi|, |\eta|>r \;\wedge\; \xi -\eta\in B(0,R)
\Rightarrow|a(\xi)-a(\eta)|<\eps,
\end{equation}
where $B(0,R)\subset \R^d$ is the ball centered in zero with the
radius $R$.}

Here, we want conditions that are more intuitive than \eqref{uvjet}.
They are given by the following definition.
\begin{definition}
Let $\Omega\subset \R^d$ be an arbitrary open subset of the
Euclidean space $\R^d$. We say that the set $\Omega$ admits a
complete fibration along the family of curves (below, $I$ denotes a
set of indices)
\begin{equation*}
{\cal C}=\{\varphi_\lambda: \R^+\to \Omega: \; \lambda\in I \}
\end{equation*} if for every $x\in \Omega$ there exist a unique
$t\in \R^+$ and unique $\lambda\in I$ such that
$x=\varphi_\lambda(t)$.
\end{definition}

Assume that we have a family of curves
\begin{equation}
\label{famofcurv} {\cal C}=\{\varphi_\lambda: \R^+\to \R^d: \;
\varphi_\lambda(t)=t\psi_{{\lambda}}(t); \; \lambda,
\psi_{{\lambda}}(t) \in S^{d-1};\,\psi_{{\lambda}}(1)=\lambda \},
\end{equation} parameterized by the distance of the origin, which
completely fibrates $\R^d\setminus\{0\}$. We have chosen the unit
sphere $S^{d-1}$ intentionally since we would like $\lambda\in
S^{d-1}$ to determine the "direction" of the curve
$\varphi_\lambda$.

Furthermore, assume that there exist a constant $c>0$ and an
increasing real function $f$ satisfying $f(z)\to \infty$ as $z\to
\infty$ such that, for any $\lambda_1,\lambda_2\in S^{d-1}$ and any
$t_1,t_2\in \R^+$, it holds:
\begin{equation}
\label{id_1}
|t_1\psi_{{\lambda}_1}(t_1)-t_2\psi_{{\lambda}_2}(t_2)|\geq
cf(\min\{t_1,t_2 \})|\lambda_1-\lambda_2|,
\end{equation}where $\psi_{{\lambda}}$ are defined in \eqref{famofcurv}.

Finally, let $a\in L^\infty(\R^d)$ and $a_\infty\in C(S^{d-1})$ be
functions such that:
\begin{equation}
\label{id_2} \lim\limits_{t\to
\infty}a(\varphi_\lambda(t))=a_\infty(\lambda), \ \ {\rm uniformly \
in} \ \lambda\in S^{d-1},
\end{equation}and let $b:\R^d\to \R$ be a continuous
function converging to zero at infinity. We associate to $a$ and $b$
operators ${\cal A}$ and $B$, respectively, as defined in
\eqref{oper_1} and \eqref{oper_2}. The following commutation lemma
holds.

\begin{lemma}
\label{scl'} The operator $C={\cal A}B-B{\cal A}$ is a compact
operator from $L^2(\R^d)$ into $L^{2}(\R^d)$.
\end{lemma}
\begin{proof}
The proof initially follows steps from the proof of Tartar's First
commutation lemma. On the first step notice that we can assume $b\in
C^1_0(\R^d)$. Indeed, if we assume merely $b\in C_0(\R^d)$ then we
can uniformly approach the function $b$ by a sequence $(b_n)\in
C^1_0(\R^d)$ such that for every $n\in \N$ the function $\F(b_n)$
has a compact support. The corresponding sequence of commutators
$C_n={\cal A} B_n-B_n {\cal A}$, where $B_n(u)=b_n u$, converges in
norm toward $C$. So, if we prove that $C_n$ are compact for each
$n$, the same will hold for $C$ as well. Then, consider the Fourier
transform of the operator ${\cal C}$. It holds:
$$
\F({\cal C}u)(\xi)=\int_{\R^d}\F b(\xi-\eta)\left(a(\xi)-a(\eta)
\right)\F u(\eta) d\eta.
$$ So, following the proof of \cite[Lemma 1.7]{Tar} (or directly from \cite[Lemma
28.2]{tar_book}), to complete the proof of our lemma, it is enough
to prove \eqref{uvjet}.

First, notice that for all $\xi,\eta\in \R^d\setminus\{0\}$ such
that $\xi=\varphi_{\lambda_1}(t_1),\,\eta=\varphi_{\lambda_2}(t_2)$,
we have according to \eqref{id_1}
\begin{equation}
\label{id_3} |\lambda_1-\lambda_2|\leq \frac{|\xi-\eta|}{c
f(\min\{|\xi|,|\eta|\})}.
\end{equation} Now, let $M>0$ and $\eps>0$ be arbitrary, and
let $\xi,\eta\in \R^d\setminus\{0\}$ be such that $\xi-\eta\in
B(0,M)$.

According to our assumptions from Definition 4, there are unique
$\lambda_1,\lambda_2\in S^{d-1}$ and $t_1,t_2\in \R^+$ such that
$\xi=\varphi_{\lambda_1}(t_1),\;\eta=\varphi_{\lambda_2}(t_2)$.
Second, $S^{d-1}$ is compact, and so $a_\infty$ is uniformly
continuous:
$$(\exists \delta>0)\;|\lambda_1-\lambda_2|<\delta \Rightarrow |a_\infty(\lambda_1)-a_\infty(\lambda_2)|<\frac{\eps}{3}.$$
Third, according to \eqref{id_2} there is $R_1>0$ such that
$$t_1,t_2>R_1\Rightarrow|a(\xi)-a_\infty(\lambda_1)|<\frac{\eps}{3}\mbox{
and } |a(\eta)-a_\infty(\lambda_2)|<\frac{\eps}{3}.$$

Finally, \eqref{id_3} imply $$(\exists
R_2>0)\;t_1,t_2>R_2\Rightarrow|\lambda_1-\lambda_2|\leq\frac{|\xi-\eta|}{cf(R_2)}\leq\frac{\mbox{diam}K}{cf(R_2)}<\delta,$$
and so for $R=\max\{R_1,R_2\},\;|\xi|,|\eta|\geq R\mbox{ and
}\xi-\eta\in B(0,M)$ we have
$$|a(\xi)-a(\eta)|\leq|a(\xi)-a_\infty(\lambda_1)|+|a_\infty(\lambda_1)-a_\infty(\lambda_2)|+|a_\infty(\lambda_2)-a(\eta)|<\frac{\eps}{3}+\frac{\eps}{3}
+\frac{\eps}{3}=\eps.$$

The proof is over. \end{proof}

\begin{definition}
\label{novo_1} We say that a manifold $P\subset \R^d$ is admissible
if there exists a fibration of the space $\R^d\backslash \{0\}$
along a family of curves ${\cal C}$ of form \eqref{famofcurv}
 such that for every $y\in P$ there exists a unique
$\varphi_{\lambda(y)} \in {\cal C}$ such that $y\in
\{\varphi_{\lambda(y)}(t):\; t\in \R^+ \}$ and $\R^d\backslash
\{0\}=\dot\bigcup_{y\in P} \{\varphi_{\lambda(y)}(t):\; t\in \R^+
\}$, where $\dot\bigcup$ denotes the disjoint union.

We say that the function $\tilde{\psi}\in C(\R^d)$ is an admissible
symbol if for every $\varphi_\lambda\in {\cal C}$ it holds $
\lim\limits_{t\to \infty}\tilde{\psi}(\varphi_\lambda(t))=\psi(y)$,
where $y\in P$ is such that $y\in \{\varphi_\lambda(t):\; t\in \R^+
\}$, $\lambda\in S^{d-1}$, and $\psi\in C(P)$.

We shall also write
\begin{equation}
\label{admi_symb} \lim\limits_{\xi\to
\infty}(\tilde{\psi}-(\psi\circ\pi_P))(\xi)=0,
\end{equation}where $\pi_P$ is the projection of the point $\xi$ on
the manifold $P$ along the fibres ${\cal C}$.
\end{definition}

We shall define an extension of the $H$-measures on the set
$\R^d\times P$, where $P$ is a manifold admissible in the sense of
Definition \ref{novo_1}. The following theorem holds:
\begin{theorem}
\label{tbasic1_novo} Denote by $P$ a manifold admissible in the
sense of Definition \ref{novo_1}. If $(u_n)=((u_n^1,\dots, u_n^r))$
is a sequence in $L^2(\R^d;\R^r)$ such that $u_n\rightharpoonup 0$
in $L^2(\R^d;\R^r)$, then there exists its subsequence $(u_{n'})$
and a positive definite matrix of complex Radon measures
$\mu=\{\mu^{ij}\}_{i,j=1,\dots,r}$ on $\R^d\times P$ such that for
all $\varphi_1,\varphi_2\in C_0(\R^d)$ and an admissible symbol
$\tilde{\psi}\in C(\R^d)$:
\begin{equation}
\label{basic1_new}
\begin{split}
\lim\limits_{n'\to \infty}\int_{\R^d}&(\varphi_1
u^i_{n'})(x)\overline{{\cal A}_{\tilde{\psi}}(\varphi_2
u^j_{n'})(x)}dx=\langle\mu^{ij},\varphi_1\overline{\varphi_2}\psi
\rangle\\&= \int_{\R^d\times
P}\varphi_1(x)\overline{\varphi_2(x)}\psi(\xi)d\mu^{ij}(x,\xi), \ \
\xi\in P,
\end{split}
\end{equation}where ${\cal A}_{\tilde{\psi}}$ is a multiplier operator with the (admissible) symbol $\tilde{\psi}\in
C(\R^{d})$, and $\psi\in C(P)$ is such that \eqref{admi_symb} is
satisfied.
\end{theorem}
\begin{proof}

First, notice that
\begin{align}
\label{h-mjere_1} \int_{\R^d}&(\varphi_1 u^i_{n'})(x)\overline{{\cal
A}_{\tilde{\psi}}(\varphi_2
u^j_{n'})(x)}dx\\
&=\int_{\R^d}\F(\varphi_1 u^i_{n'})(\xi)\overline{\F(\varphi_2
u^j_{n'})(\xi)}\tilde{\psi}(\xi)d\xi, \nonumber
\end{align}according to the Plancherel theorem. Then, denote by
$\pi_P(x)$ the projection of the point $x\in \R^d$ on the manifold
$P$ along the corresponding fibres. It holds
\begin{align}
\label{h-mjere_2} &\int_{\R^d}\F(\varphi_1
u^i_{n'})(\xi)\overline{\F(\varphi_2
u^j_{n'})(\xi)}\tilde{\psi}(\xi)d\xi\\
&=\int_{\R^d}\F(\varphi_1 u^i_{n'})(\xi)\overline{\F(\varphi_2
u^j_{n'})(\xi)}(\psi\circ\pi_P)(\xi))d\xi\nonumber\\&+
\int_{\R^d}\F(\varphi_1 u^i_{n'})(\xi)\overline{\F(\varphi_2
u^j_{n'})(\xi)}\left(\tilde{\psi}(\xi)-(\psi\circ\pi_P)(\xi)\right)d\xi.
\nonumber
\end{align}From the fact that the symbol $\tilde{\psi}$ is admissible in the sense of Definition \ref{novo_1} and the Lebesgue dominated converges theorem,
it follows
$$
\lim\limits_{n'\to \infty}\int_{\R^d}\F(\varphi_1
u^i_{n'})(\xi)\overline{\F(\varphi_2
u^j_{n'})(\xi)}\left(\tilde{\psi}(\xi)-(\psi\circ\pi_P)(\xi)\right)d\xi=0.
$$ From here,  \eqref{h-mjere_1} and \eqref{h-mjere_2}, we conclude
\begin{align*}
&\lim\limits_{n'\to \infty}\int_{\R^d}(\varphi_1
u^i_{n'})(x)\overline{{\cal A}_{\tilde{\psi}}(\varphi_2
u^j_{n'})(x)}dx\\
&=\lim\limits_{n'\to \infty}\int_{\R^d}(\varphi_1
u^i_{n'})(x)\overline{{\cal A}_{\psi\circ \pi_P}(\varphi_2
u^j_{n'})(x)}dx, \nonumber
\end{align*} implying that, in order to prove \eqref{basic1_new}, it is enough
to prove it for the multipliers with symbols defined on $P$. Now,
the proof completely follows the one of \cite[Theorem 1.1]{Tar}. Let
us briefly recall it.

Notice that, according to the first commutation lemma (Lemma
\ref{scl'}), the mapping
$$
(\varphi_1\overline{\varphi_2},\psi)\mapsto \int_{\R^d}(\varphi_1
u^i_{n'})(x)\overline{{\cal A}_\psi(\varphi_2 u^j_{n'})(x)}dx
$$ is a positive bilinear functional on $C_0(\R^d)\times  C(P)$. According to the Schwartz kernel theorem, the functional
can be extended to a continuous linear functional on ${\cal
D}(\R^d\times P)$. Since it is positive, due to the Schwartz lemma
on non-negative distributions, it follows that the mentioned
extension is a Radon measure.
\end{proof}

\begin{remark}
\label{rem_tar} If we assume that the sequence $(u_n)$ defining the
$H$-measure is bounded in $L^p(\R^d)$ for $p>2$, then we can take
the test functions $\varphi_1, \varphi_2$ from Theorem \ref{tbasic1}
such that $\varphi_1\in L^q(\R^d)$ where $1/q+2/p\leq 1$, and
$\varphi_2\in C_0(\R^d)$ (see \cite[Corollary 1.4]{Tar} and
\cite[Remark 2, a)]{pan_arma}).
\end{remark}

\section{Strong precompactness property of a sequence of quasisolutions to a fractional scalar conservation law}

Differential equations involving fractional derivatives have
received considerable amount of attention recently (see e.g.
\cite{ali1, droniou} and references therein). Here, we shall
consider a sequence of quasi-solutions to a (purely) fractional
scalar conservation law. The definition of a quasi-solution for a
classical conservation law can be found in \cite[Definition
1.2]{pan_jhde}. It actually represents a slightly relaxed version of
Kru\v zkov's admissibility conditions \cite{Kru}. Among other facts,
the mentioned conditions are obtained relying on the Leibnitz rule
for the derivatives of product. This rule does not hold for the
fractional derivatives. Therefore, we need to modify slightly
Panov's definition of quasisolutions. The motivation for the
modification lies in the procedure from \cite{Sazh} (see also
\cite{MA1}) where the existence of solution to an ultra-parabolic
equation is proved relying on the $H$-measures and compactness of
appropriate operators.

\begin{definition}
\label{q-sol} We say that a function $u\in L^\infty(\R^d)$ is a
quasisolution to
 equation \eqref{eq_frac}
 if for every $\lambda\in \R$, $\varphi_1\in
C^\infty_c(\R^d)$ and $\varphi_2\in L^\infty(\R^d)$, it holds
\begin{equation}
\label{novo_4} \begin{split} \int_{\R^d}\!\sum\limits_{k=1}^d\!{\rm
sgn}(u\!-\!\lambda)(f_k(x,u)\!-\!f_k(x,\lambda))\varphi_1(x)\overline{{\cal
A}_{\frac{(i\xi_k)^{\alpha_k}}{|\xi_1|^{\alpha_1}\!+\!|\xi_2|^{\alpha_2}\!+\!\dots\!+\!|\xi_d|^{\alpha_d}}}\varphi_2
(x)}dx&\\=\int_{\R^d}L_{\lambda,\varphi_1}[\varphi_2]dx&,
\end{split}
\end{equation} where

\begin{itemize}

\item ${\cal A}_{\frac{(i\xi_k)^{\alpha_k}}{|\xi_1|^{\alpha_1}+|\xi_2|^{\alpha_2}+\dots+|\xi_d|^{\alpha_d}
}}$ is a multiplier
operator with the symbol \\
$\frac{(i\xi_k)^{\alpha_k}}{|\xi_1|^{\alpha_1}+|\xi_2|^{\alpha_2}+\dots+|\xi_d|^{\alpha_d}}$;

\item the linear operator
$L_{\lambda,\varphi_1}: L^\infty(\R^d)\to L^1(\R^d)$ is compact.

\end{itemize}

\end{definition}

{ The operator $L_{\lambda,\varphi_1}$ we call an {\bf entropy
defect operator}. In the case of classical scalar conservation laws,
the operators $L_{\lambda,\varphi_1}$, $\lambda\in \R$, will
correspond to the appropriate entropy defect measures weighted by
$\varphi_1{\cal A}_{\frac{1}{|\xi|}}(\cdot)$, where ${\cal
A}_{\frac{1}{|\xi|}}$ is the multiplier operator with the symbol
$\frac{1}{|\xi|}$.}

An interesting question might be how to define a weak solution to
\eqref{eq_frac} analog to the standard weak solution for a PDE of an
integer order. Let us recall how one can (formally) reach to a
definition of weak solution for a first order partial differential
equation.

So, for a function $f(x,\lambda)=(f_1(x,\lambda),\dots,
f_d(x,\lambda))\in BV(\R^d;C(\R))$, $(x,\lambda)\in \R^d\times \R$,
consider
$$
{\rm div} f(x,u)=0, \ \ u\in L^\infty(\R^d).
$$ Finding the Fourier transform of the last expression, we obtain
\begin{equation}
\label{motiv_1} \sum\limits_{k=1}^di\xi_k \F(f_k(\cdot,u))(\xi)=0, \
\ \xi\in \R^d.
\end{equation}
Then, take an arbitrary function $\varphi\in C^1_c(\R^d)$ and
multiply \eqref{motiv_1} by $\overline{\F(\varphi)}(\xi)$ (inverse
Fourier transform of $\varphi$). We obtain
\begin{align*}
&\sum\limits_{k=1}^di\xi_k
\F(f_k(\cdot,u))(\xi)\overline{\F(\varphi)}(\xi)=-\sum\limits_{k=1}^d
\F(f_k(\cdot,u))(\xi)\overline{i\xi_k
\F(\varphi)}(\xi)\\&=-\sum\limits_{k=1}^d
\F(f_k(\cdot,u))(\xi)\overline{ \F(\pa_{x_k}\varphi)}(\xi)=0.
\end{align*}Integrating this over $\xi\in \R^d$ and applying the
Plancherel formula, we get
$$
-\int_{\R^d}\sum\limits_{k=1}^d \F(f_k(\cdot,u))(\xi)\overline{
\F(\pa_{x_k}\varphi)}(\xi)d\xi=-\int_{\R^d}\sum\limits_{k=1}^d
f_k(x,u){\pa_{x_k}\varphi(x)} dx=0,
$$ which is the classical definition of a weak solution.

From the latter considerations, it is natural to define an
integrable function $u$ to be a weak solution to \eqref{eq_frac} if
for every $\varphi\in C_c^\infty(\R^d)$, it holds
\begin{equation*}
\int_{\R^d}\sum\limits_{k=1}^d
f_k(x,u(x))\overline{\pa^{\alpha_k}_{x_k} \varphi (x)}dx=0,
\end{equation*} where $\pa^{\alpha_k}_{x_k}$ is the multiplier
operator with the symbol $(i\xi_k)^{\alpha_k}$, $k=1,\dots,d$.

Existence of a sequence of quasisolutions to \eqref{eq_frac} is an
open question which we will deal with in a future. Existence of the
sequence of quasisolutions together with the strong precompactness
result (Theorem \ref{th_frac}) would immediately give existence of a
weak solution to \eqref{eq_frac}.

The latter notion of quasisolution can be rewritten in the so called
kinetic formulation which appeared to be very powerful in the field
of conservation laws \cite{Lio3}. It reduces equation
\eqref{eq_frac} to a linear equation with the right-hand side in the
form of a distribution of order one.

It is enough to find derivative in $\lambda$ to \eqref{novo_4}.
Thus, in the sense of distributions, we have
\begin{align}
\label{pro-1'} &-\int_{\R^d}\sum\limits_{k=1}^d h(x,\lambda)
\pa_\lambda f_k(x,\lambda)
 \varphi_1(x)\overline{{\cal A}_{\frac{(i\xi_k)^{\alpha_k}}{|\xi_1|^{\alpha_1}+|\xi_2|^{\alpha_2}+\dots+|\xi_d|^{\alpha_d}
}}\varphi_2 (x)}dx\\&=\int_{\R^d}\pa_\lambda
L_{\lambda,\varphi_1}[\varphi_2]dx, \nonumber
\end{align} where $h(x,\lambda)={\rm
sgn}(u(x)-\lambda)$, or equivalently, for any $\rho\in C^1_0(\R)$
\begin{align}
\label{pro-1}
&\int_{\R}\int_{\R^d}\sum\limits_{k=1}^dh(x,\lambda)\pa_\lambda
f(x,\lambda)\rho(\lambda)\varphi_1(x)\overline{{\cal
A}_{\frac{(i\xi_k)^{\alpha_k}}{|\xi_1|^{\alpha_1}+|\xi_2|^{\alpha_2}+\dots+|\xi_d|^{\alpha_d}
}}\varphi_2 (x)}dx d\lambda\\&=\int_{\R}\int_{\R^d}
L_{\lambda,\varphi_1}[\varphi_2] \rho'(\lambda) dx d\lambda.
\nonumber
\end{align}

We shall prove that under a genuine nonlinearity condition for the
flux function $f(x,\lambda)=(f_1(x,\lambda),\dots,f_d(x,\lambda))$
from the previous definition, a sequence of quasisolutions to
\eqref{eq_frac} is strongly precompact in $L^1_{loc}(\R^d)$.

\begin{definition}
\label{non=deg} We say that equation \eqref{eq_frac} is genuinely
nonlinear if for almost every $x\in {\bf R}^d$ the mapping
\begin{gather}
\lambda \mapsto \sum\limits_{k=1}^d (i\xi_k)^{\alpha_k}
f_k(x,\lambda),
 \label{nondeg}
\end{gather} where $i$ is the imaginary unit, is not identically equal to zero on any set of positive measure
$X\subset \R$.
\end{definition}

To continue, denote by $P=\{\xi\in \R^d:\;
\sum\limits_{k=1}^d|\xi_k|^{\alpha_k}=1 \}$ where $\alpha_k$,
$k=1,\dots,d$, are given in \eqref{eq_frac}. Notice that the
manifold $P$ is admissible manifold in the sense of Definition
\ref{novo_1}. For the family ${\cal C}$ from Definition \ref{novo_1}
corresponding to the manifold $P$, we will take the family of curves
defined by

\begin{equation}
\label{prop-f} \xi_k(t)=\eta_k t^{1/\alpha_k}, \ \  t\geq 0, \ \
k=1,\dots,d, \ \ (\eta_1,\dots,\eta_d)\in P
\end{equation} Therefore, there exists an $H$-measure $\mu$ defined on $\R^d\times
P$ as given in Theorem \ref{tbasic1_novo}.

\begin{remark}
Remark that there can be several manifolds (compare \cite{Ant} and
\cite{Ant2} in the parabolic case) as well as several fibrations
that we could use. If we need a smoother manifold, we could take
$\tilde{P}=\{\xi\in \R^d:\;
\left(\sum\limits_{k=1}^d|\xi_k|^{2\alpha_k}\right)^{1/2}=1 \}$.
Also, we can take several fibrations, but the one that should be
used here is exactly \eqref{prop-f} since in that case the symbols
$\frac{(i\xi_k)^{\alpha_k}}{|\xi_1|^{\alpha_1}+|\xi_2|^{\alpha_2}+\dots+|\xi_d|^{\alpha_d}}$,
 $k=1,\dots,d$, are admissible test functions in \eqref{basic2_new} and we can pass to the limit as $n'\to \infty$ in \eqref{pro-5}.
We would like to thank to E.Yu.Panov for helping us to clear up this
situation.
\end{remark}

To proceed, denote by $(u_n)$ a family of quasi-solutions to
\eqref{eq_frac} satisfying the non-degeneracy condition in the sense
of Definition \ref{non=deg}. The following theorem holds:

\begin{theorem}
\label{th_frac} Let $(u_n)$ be a bounded sequence of quasi-solutions
to \eqref{eq_frac}. Assume that there exists a subsequence (not
relabeled) $(u_n)$ of the given sequence such that, for every
$\lambda \in \R$ and $\varphi_1\in C^\infty_c(\R^d)$, the
corresponding sequence of entropy defects operators
$(L^n_{\lambda,\varphi_1})$ admits a limit in the sense that there
exists a compact operator $L_{\lambda,\varphi_1}: L^\infty(\R^d)\to
L^1(\R^d)$ such that for any $\rho\in C^1_0(\R)$ and any sequence
$(\varphi_n)$ weakly-$\star$ converging to zero in $L^\infty(\R^d)$,
it holds
$$
\lim\limits_{n\to
\infty}\int_{\R}\int_{\R^d}\left(L^n_{\lambda,\varphi_1}[\varphi_n]-L_{\lambda,\varphi_1}[\varphi_n]\right)\rho(\lambda)dx
d\lambda=0.
$$

Then, the sequence $(u_n)$ is strongly precompact in
$L^1_{loc}(\R^d)$.
\end{theorem}

Notice that we have the situation from the latter theorem in the
case of a classical scalar conservation law { (see e.g. \cite{MA1,
NHM_mit} and the comments after Definition \ref{q-sol})}.

{ Denote
\begin{equation}
\label{sgn} h_n(x,\lambda)={\rm sgn}(u_n(x)-\lambda)
\end{equation}and assume that for a function $h\in L^\infty(\R^d\times\R)$, it holds
\begin{equation}
\label{limit} h_n(x,\lambda) \def\Dscon{\relbar\joinrel\dscon}
\povrhsk\ast h(x,\lambda) \ \ {\rm in} \ \ L^\infty(\R^d)
\end{equation} along a subsequence of the
sequence $(h_n)$. } Taking Remark \ref{rem_tar} into account, the
following extension of Theorem \ref{tbasic1_novo} can be proved in
the exactly same way as \cite[Theorem 3]{msb95} (see also
\cite[Remark 2, a)]{pan_arma}, \cite[Theorem N]{Sazh},
\cite[Proposition 2]{JMSpa}):
\begin{theorem}
\label{tbasic1_novo'} 1. { For the sequence $(h_n)$ and the function
$h$ defined by \eqref{sgn} and \eqref{limit}, respectively,} there
exists a set $E\subset \R$ of a full measure such that there exists
a family of complex Radon measures $\mu=\{\mu^{pq}\}_{p,q\in E}$ on
$\R^d\times P$ such that there exists a subsequence $(h_{n'}-h)$ of
the sequence $(h_{n}-h)$ such that for all $\varphi_1\in L^2(\R^d)$,
$\varphi_2\in C_c(\R^d)$ and a symbol $\psi\in C(\R^d)$ admissible
in the sense of Definition \ref{novo_1}:
\begin{equation}
\label{basic2_new}
\begin{split}
\lim\limits_{n'\to \infty}\int_{\R^d}&\varphi_1(x)
(h_{n'}-h)(x,p)\overline{{\cal A}_\psi(\varphi_2(\cdot)
(h_{n'}-h)(\cdot,q))(x)}dx\\&=\langle\mu^{pq},\varphi_1\overline{\varphi_2}\psi\circ
\pi_P \rangle= \int_{\R^d\times
P}\varphi_1(x)\overline{\varphi_2(x)}\psi\circ
\pi_P(\xi)d\mu^{pq}(x,\xi),
\end{split}
\end{equation}where $(x,\xi)\in \R^d\times P$, and ${\cal A}_\psi$ is a multiplier operator with the (admissible) symbol $\psi\in C(\R^{d})$.

2. The mapping $(p,q)\mapsto \mu^{pq}$ as the mapping from $E\times
E$ to the space ${\cal M}(\R^d\times P)$ of complex Radon measures
is continuous with the topology generated by the semi-norms
$||\mu||_K=Var(\mu)(K)$, $K$-compact in $\R^d\times P$.

\end{theorem}

Now, we can prove Theorem \ref{th_frac}.

{\bf Proof of Theorem \ref{th_frac}:} The proof uses the kinetic
formulation \eqref{pro-1} of \eqref{novo_4}.

{ First, take the functions $h_n$ and $h$ defined by \eqref{sgn} and
\eqref{limit}, respectively.} Then, notice that according to
\eqref{pro-1}, a subsequence $(h_{n'}-h)$ of the sequence
$(h_{n}-h)$ given in Theorem \ref{tbasic1_novo'} satisfies
\begin{equation}
\label{pro-5}
\begin{split}
&\int_{\R}\int_{\R^d}\!\sum\limits_{k=1}^d\!(h_{n'}\!-\!h)(x,\lambda)\pa_\lambda
f_k(x,\lambda)\rho(\lambda)\varphi_1(x)\overline{{\cal
A}_{\frac{(i\xi_k)^{\alpha_k}}{|\xi_1|^{\alpha_1}\!+\!|\xi_2|^{\alpha_2}\!+\!\dots\!+\!|\xi_d|^{\alpha_d}}}\varphi_2
(x)}dx\\&=\int_{\R}\int_{\R^d}(L^{n'}_{\lambda,\varphi_1}[\varphi_2]-L_{\lambda,\varphi_1}[\varphi_2])\rho'(\lambda)dxd\lambda,
\end{split}
\end{equation} where $\rho\in C^1_c(\R)$, $\varphi_1\in
C^\infty_c(\R^d)$, and $\varphi_2\in L^\infty(\R^d)$ are arbitrary.
Then, for a fixed $p\in \R$, put
$$
\varphi_2(x)=\varphi_2^{n'}(x,p)=(h_{n'}-h)(x,p)\phi_2(x,p), \ \
\phi_2\in C^\infty_c(\R^d\times\R).
$$After letting $n'\to \infty$ in \eqref{pro-5}, from Theorem
\ref{tbasic1_novo'} and conditions on $L^n_{\lambda,\varphi_1}$ and
$L_{\lambda,\varphi_1}$ given in Theorem \ref{th_frac}, we conclude
that for almost every $p\in \R$:
\begin{align*}
\int_{\R}\int_{\R^d\times P}\sum\limits_{i=1}^d\xi^{\alpha_k}_i
f_i(x,\lambda)\varphi_1(x)\rho(\lambda)\phi_2(x,p)d\mu^{\lambda
p}(x,\xi)dp =0,
\end{align*} where $\mu$ is an $H$-measure corresponding to the sequence $(h_n-h)$, as given in Theorem \ref{tbasic1_novo'}. For a fixed $q$, put here
$\rho(\lambda)=\frac{1}{\eps}\tilde{\rho}(\frac{\lambda-q}{\eps})\bar{\rho}(\frac{q+\lambda}{2})$
and $\phi_2(x,p)=\frac{1}{\eps}\tilde{\rho}(\frac{p-q}{\eps})
\bar{\rho}(\frac{q+p}{2})\varphi_1(x)$, where $\tilde{\rho}$ is a
non-negative compactly supported real function with total mass one,
and $\bar{\rho}\in C^1_0(\R)$ is arbitrary. Integrating over $p,q\in
\R$, and letting $\eps\to 0$, we obtain:
\begin{align*}
\int_{\R}\int_{\R^d\times P}\sum\limits_{k=1}^d (i\xi_k)^{\alpha_k}
f_k(x,q)\varphi_1^2(x)\bar{\rho}(q)d\mu^{qq}(x,\xi)dq=0.
\end{align*} From the genuine nonlinearity condition, we conclude
$\mu^{\lambda\lambda}\equiv 0$ for almost every $\lambda\in E$ (see
e.g. \cite[Theorem 5]{msb95}). This actually means that $h_{n'}\to
h$ strongly in $L^2_{loc}(\R^d\times \R)$, and that
$h(x,\lambda)={\rm sgn}(u(x)-\lambda)$ for some $u\in
L^\infty(\R^d)$. From here, it is not difficult to conclude that
$u_{n'}\to u$ strongly in $L^1_{loc}(\R^d)$. This concludes the
proof. $\Box$






\begin{thebibliography}{99}

\bibitem{MA1} (MR2563677)
\newblock J.~Aleksi\'c, D.~Mitrov\'ic, S.~Pilipovi\'c,
\newblock \emph{Hyperbolic conservation laws with vanishing nonlinear
diffusion and linear dispersion in heterogeneous media},
\newblock Journal of Evolution Equations, {\bf 9} (2009), 809--828.


\bibitem{ali1} (MR2305729)
\newblock N. Alibaud,
\newblock \emph{Entropy formulation for fractal conservation laws},
\newblock Journal of Evolution Equations, {\bf 7} (2007), 145--175.


\bibitem{Ant} (MR2728569)
\newblock N.~Antonic, M.~Lazar,
\newblock \emph{Parabolic variant of H-measures in
homogenisation of a model problem based on Navier-Stokes equation},
\newblock Nonlinear Analysis---Real World Appl, {\bf 11} (2010), 4500--4512.


\bibitem{Ant2} (MR2412122)
\newblock N.~Antonic, M.~Lazar,
\newblock \emph{ $H$-measures and variants applied to parbolic equations},
\newblock J. Math. Anal. Appl., {\bf 343} (2008), 207--225.


\bibitem{Dpe} (MR0808729)
\newblock R.~DiPerna,
\newblock \emph{Compensated compactness and general systems of conservation
laws},
\newblock Trans. Amer. Math. Soc., {\bf 292} (1985), 383-419.


\bibitem{droniou} (MR2259335)
\newblock J. Droniou and C. Imbert,
\newblock \emph{Fractal first-order partial differential equations},
\newblock Arch. Ration. Mech. Anal., {\bf 182} (2006), 299--331.

\bibitem{Ger} (MR1135919)
\newblock Gerard, P.,
\newblock \emph{Microlocal Defect Measures},
\newblock Comm. Partial Differential Equations {\bf 16} (1991), 1761--1794.


\bibitem{Kru}
\newblock S.~N.~Kruzhkov,
\newblock \emph{First order quasilinear equations in several
independent variables},
\newblock Mat. Sbornik. {\bf 81} (1970), 228--255;
English transl. in Math. USSR Sb. {\bf 10} (1970), 217--243.


\bibitem{Lio3} (MR1201239)
\newblock P.~L.~Lions, B.~Perthame, E.~Tadmor,
\newblock \emph{A kinetic formulation of multidimensional scalar conservation laws and related
equations},
\newblock J. of American Math. Soc.,  {\bf 7} (1994), 169--191.

\bibitem{Lio1} (MR0778970 )
\newblock P.~L.~Lions,
\newblock \emph{The concentration compactness principle in the calculus of variations. The locally compact
case.}
\newblock Part I. Ann. Inst. H.Poincare Sect. A (N.S.), {\bf 1}
(1984), 109--145, 223--283.


\bibitem{Lio2} (MR0778974 )
\newblock P.~L.~Lions,
\newblock \emph{The concentration compactness principle in the calculus of variations. The locally compact
case},
\newblock Part II. Ann. Inst. H.Poincare Sect. A (N.S.), {\bf 1}
(1984), 109--145, 223--283.


\bibitem{NHM_mit} (MR2601993)
\newblock D.~Mitrovic,
\newblock \emph{Existence and stability of a multidimensional scalar
conservation law with discontinuous flux},
\newblock Netw. Het. Media, {\bf 5} (2010), 163--188.



\bibitem{msb95} (MR1264775)
\newblock E.~Yu.~Panov,
\newblock \emph{On sequences of measure-valued solutions of a first order
quasilinear equations},
\newblock Russian Acad. Sci. Sb. Math. Vol. {\bf 81}
(1995), 211--227.



\bibitem{JMSpa} (MR2544036)
\newblock E.Yu.~Panov,
\newblock \emph{Ultra-parabolic equations with rough coefficients.
Entropy solutions and strong pre-compactness property},
\newblock  Journal of Mathematical Sciences, {\bf 159} (2009) 180--228.

\bibitem{pan_jhde} (MR2374223)
\newblock E.~Yu.~Panov,
\newblock \emph{Existence of strong traces for quasi-solutions of multidimensional conservation laws},
\newblock Journal of Hyperbolic Differential Equations, {\bf 4}
(2007), 729--770.

\bibitem{pan_arma} (MR2592291)
\newblock E. Yu. Panov,
\newblock \emph{Existence and strong precompactness properties for
entropy solutions of a first-order quasilinear equation with
discontinuous flux},
\newblock Arch. Ration. Mech. Anal.
{\bf 195} (2010) 643–-673.

\bibitem{Sazh} (MR2227988)
\newblock Sazhenkov, S. A.,
\newblock \emph{The genuinely nonlinear Graetz-Nusselt ultraparabolic
equation},
\newblock (Russian. Russian summary) Sibirsk. Mat. Zh. 47 (2006),
no. 2, 431--454; translation in Siberian Math. J. {\bf 47} (2006),
355--375


\bibitem{Tar} (MR1069518)
\newblock L.~Tartar,
\newblock \emph{H-measures, a new approach for studying homogenisation,
oscillation and concentration effects in PDEs},
\newblock Proc. Roy. Soc.
Edinburgh. Sect. A {\bf 115} (1990) 193-230


\bibitem{tar_book} (MR2582099)
\newblock L.~Tartar,
\newblock \emph{The general theory of homogenization. A personalized
introduction},
\newblock Lecture Notes of the Unione Matematica Italiana, 7. Springer-Verlag,
Berlin; UMI, Bologna, 2009. xxii+470 pp.

\end{thebibliography}
\end{document}